\documentclass[]{amsart}
\usepackage[hyperindex=true,bookmarks=true,bookmarksnumbered=true]{hyperref}
\usepackage{amssymb}
\usepackage{graphicx}
\usepackage{epsfig}
\usepackage{verbatim}
\usepackage[all]{xy}
\allowdisplaybreaks

\def\undersetbrace#1\to#2{\underbrace{#2}_{#1}}
\def\oversetbrace#1\to#2{\overbrace{#2}^{#1}}
\def\AMSunderset#1\to#2{\underset{#1}{#2}}
\def\AMSoverset#1\to#2{\overset{#1}{#2}}

\newcommand{\nmb}[2]{\ifx!#1{\ref{nmb:#2}}%
\else\if.#1{\label{nmb:#2}}%
\else\if0#1{\label{nmb:#2}}%
\else{{#2}}%
\fi\fi\fi}
\swapnumbers

\newtheorem*{proposition*}{Proposition}
\newtheorem{theorem}[subsection]{Theorem}
\newtheorem*{theorem*}{Theorem}
\newtheorem{lemma}[subsection]{Lemma}
\newtheorem*{lemma*}{Lemma}

\newtheorem*{corollary*}{Corollary}
\newenvironment{demo}[1]{\par\smallskip\noindent{\bf #1.}}{\par\smallskip}

\def\ign#1{}             
\def\o{\on{\circ}}
\def\X{\mathfrak X}

\def\be{\beta}

\def\de{\delta}

\def\et{\eta}

\def\ka{\kappa}
\def\la{\lambda}

\def\si{\sigma}

\def\ph{\varphi}

\def\ps{\psi}
\def\om{\omega}

\def\Om{\Omega}
\def\i{^{-1}}
\def\x{\times}
\def\p{\partial}
\def\g{\mathfrak g}
\let\on=\operatorname
\def\L{\mathcal L}
\def\Ad{\on{Ad}}
\def\ad{\on{ad}}
\def\pr{\on{pr}}
\def\Id{\on{Id}}
\def\Evol{\on{Evol}}
\def\evol{\on{evol}}
\def\AMSonly#1{}
\def\var{\quad} 

\begin{document}
\title[]
{Geometry of infinite dimensional Cartan Developments 
}
\author{Johanna Michor and Peter W. Michor}
\address{
Peter W. Michor:
Fakult\"at f\"ur Mathematik, Universit\"at Wien,
Oskar-Morgenstern-Platz 1, A-1090 Wien, Austria.
}
\email{Peter.Michor@univie.ac.at}
\address{Johanna Michor: 
Fakult\"at f\"ur Mathematik, Universit\"at Wien,
Oskar-Morgenstern-Platz 1, A-1090 Wien, Austria.
}
\email{Johanna Michor@univie.ac.at}
\thanks{Johanna Michor was supported by FWF-Project P31651}

\begin{abstract} 
The Cartan development takes a Lie algebra valued 1-form satisfying the Maurer-Cartan equation on a simply connected manifold $M$ to a smooth mapping from $M$ into the Lie group. 
In this paper this is generalized to infinite dimensional $M$ for infinite dimensional regular Lie groups. 
The Cartan development is viewed as a generalization of the evolution map of a regular Lie group.
The tangent mapping of a Cartan development is identified as another Cartan development.
\end{abstract}
\def\LaTeXonly{}

\maketitle

\section{\label{nmb:1}Introduction}
A regular Lie group $G$ is one where where one can integrate smooth curves in the Lie algebra $\g$ to smooth curves in the Lie group in a smooth way: The evolution mapping $\Evol: C^{\infty}(\mathbb R,\g) \to C^{\infty}(\mathbb R,G)$ exists and is smooth; see \ref{nmb:2.4} below. This notion is relevant for infinite dimensional Lie groups where existence results for ODEs are not available in general. 
A stronger version of this notion is due to Omori et al. \cite{OmoriMaedaYoshioka80,OmoriMaedaYoshioka81,OmoriMaedaYoshioka81b,OmoriMaedaYoshioka82,OmoriMaedaYoshioka83,OmoriMaedaYoshiokaKobayashi83} and was weakened to this version for Fr\'echet Lie groups by Milnor \cite{Milnor84}. It was generalized to Lie groups modeled on convenient vector spaces (i.e., locally convex vector spaces where each Mackey Cauchy sequence converges) in \cite{KM97r}, see also \cite{KM97}. Up to now, no example of a non-regular Lie group modelled on convenient vector spaces is known. In section \ref{nmb:2} we review regular Lie groups in order to fix notation and to make clear the analogy to the Cartan-development.

One can extend the notion of regularity to other classes of curves for certain Lie groups (e.g., modeled on Banach spaces) and ask for 
the existence of $\Evol$ on $C^{k}(\mathbb R,\g)$ for $k=0,1,2,\dots$ or even $L^1(\mathbb R,\g)$. Results in these directions are due to Gl\"ockner in \cite{Glockner12,Glockner15} and Hanusch in  \cite{Hanusch22,Hanusch20,Hanusch19}. 

In this paper we extend the domain of the evolution mapping: In order to end up in $C^{\infty}(M,G)$ for an  infinite dimensional manifold $M$ we consider the spaces $\Om^1_{\text{flat}}(M,\g)$ of flat $\g$-valued differential 1-forms on $M$ in \ref{nmb:3.1} and show that the evolution map exists and is smooth for simply connected pointed $M$, and we give an explicit description of its tangent mapping.  This is an extension to infinite dimensions of the Cartan development.  

The results of this paper could be applied in situations where a diffeomorphism group acts on a space of shapes; the diffeomorphism mapping a template shape to any shape could be interpreted as a Cartan-development. See also \cite{DiezFutakiRatiu24}. 

Infinite dimensional manifolds Lie groups are always modeled on convenient vector spaces as treated in \cite{KM97}, see also \cite{Michor20} or \cite{Wiki:convenient}

This paper was inspired by the paper \cite{FokasGelfand96} on integrable surfaces which was pointed out to the authors by 
Oleksandr Sakhnovich. In  \cite{FokasGelfand96} the term `integrable' just means that the 1-forms considered there satisfy the Maurer-Cartan formula, and it has nothing to do with complete integrability: There is no Hamitonian aspect to this, as was pointed out to us by Boris Khesin. 

\section{\label{nmb:2}Review of regular Lie groups}

\subsection{\label{nmb:2.1}Notation on Lie groups}
Let $G$ be a Lie group which may be infinite dimensional with Lie
algebra $\g$.  Later we will only consider regular Lie groups, see \ref{nmb:2.4}.
Let $\mu:G\x G\to G$ be the multiplication, let $\mu_x$ be left  
translation and $\mu^y$ be right translation, 
given by $\mu_x(y)=\mu^y(x)=xy=\mu(x,y)$. We denote inversion by $\nu:G\to G$, $\nu(x)=x\i$. 
The tangent mapping 
$T_{(a,b)}\mu:T_aG\x T_bG \to T_{ab}G$ is given by 
$$T_{(a,b)}\mu.(X_a,Y_b) = T_a(\mu^b).X_a + T_b(\mu_a).Y_b$$
and $T_a\nu:T_aG\to T_{a\i}G$ is given by
$$T_a\nu = - T_e(\mu^{a\i}).T_a(\mu_{a\i}) 
= - T_e(\mu_{a\i}).T_a(\mu^{a\i}).$$

Let $L,R:\g\to \X(G)$ be the left 
and right invariant vector field mappings, given by 
$L_X(g)=T_e(\mu_g).X$ and $R_X=T_e(\mu^g).X$, respectively. 
They are related by $L_X(g)=R_{\Ad(g)X}(g)$.
Their flows are given by 
\begin{displaymath}
\on{Fl}^{L_X}_t(g)= g.\exp(tX)=\mu^{\exp(tX)}(g),\quad
\on{Fl}^{R_X}_t(g)= \exp(tX).g=\mu_{\exp(tX)}(g).\end{displaymath}

We also need the right  
Maurer-Cartan form $\ka=\ka^r\in\Om^1(G,\g)$, given by  
$\ka_x(\xi):=T_x(\mu^{x\i})\cdot \xi$. It satisfies the left 
Maurer-Cartan equation $d\ka-\tfrac12[\ka,\ka]_\wedge=0$ which follows by evaluating
 $d\ka^r$ on right invariant vector fields $R_X,R_Y$ for $X,Y\in\g$
\begin{align*}
(d\ka^r)(R_X,R_Y) &= R_X(\ka^r(R_Y)) - R_Y(\ka^r(R_X)) - \ka^r([R_X,R_Y])
\\&
= R_X(Y) - R_Y(X) + [X,Y] = 0-0 +[\ka^r(R_X),\ka^r(R_Y)]\,,
\end{align*}
where
$[\;,\;]_\wedge$ denotes the wedge product of $\g$-valued forms on
$G$ induced by the Lie bracket. Note that
$\tfrac12[\ka,\ka]_\wedge (\xi,\et) = [\ka(\xi),\ka(\et)]$.

Similarly the 
{\it left Maurer-Cartan form} $\ka^l\in\Om^1(G,\g)$ is 
given by $\ka^l_g=T_g(\mu_{g\i}):T_gG\to \g$ and it   
satisfies the right Maurer-Cartan equation 
  $d\ka^l+\tfrac12[\ka^l,\ka^l]^\wedge_\g=0$. 
  We have also $(\nu^*\ka^r)_g = -\ka^l_{g}= -Ad(g\i)\ka^r_g$.

The (exterior) derivative of the function $\Ad:G\to GL(\g)$ can be 
expressed by
\begin{displaymath}
d\Ad = \Ad.(\ad\o\ka^l) = (\ad\o \ka^r).\Ad,\end{displaymath}
since , for $c:\mathbb R\to G$ smooth with $c(0)=e$ and $c'(0)=X$ like $\exp(tX)$ if $\exp$ exists,
$$d\Ad(T\mu_g.X) = \p_t|_0 \Ad(g.c(t)) = \Ad(g)\p_t|_0\Ad(c(t)) 
= \Ad(g).\ad(\ka^l(T\mu_g.X))\,.$$
Since we shall need it we also note 
\begin{equation*}
d(\Ad\o \nu) = -(\ad\o\ka^l)(\Ad\o\nu) = -(\Ad\o\nu)(\ad\o \ka^r).
\end{equation*}

\subsection{\label{nmb:2.2}The right and left logarithmic derivatives}
Let $M$
be a manifold and let $f:M\to G$ be a smooth mapping into a Lie
group $G$ with Lie algebra $\mathfrak g$. We define the mapping 
$\de^r f:TM\to \mathfrak g$ by the formula 
\begin{align*}
\de^r f(\xi_x) :&= T_{f(x)}(\mu^{f(x)\i}).T_xf.\xi_x = \ka^r_{f(x)}(T_xf.\xi_x)
\\&=(f^*\ka^r)(\xi_x)
\text{ for }\xi_x\in T_xM.
\end{align*}
Then $\de^r f$ is a $\mathfrak g$-valued 1-form on $M$, $\de^r f\in
\Om^1(M;\mathfrak g)$. We call $\de^r f$ the
{\it right logarithmic derivative}%
\index{right logarithmic derivative} of $f$, since for $f:\mathbb
R\to(\mathbb R^+,\cdot)$ we have 
$\de^r f(x).1=\frac{f'(x)}{f(x)}=(\log\o f)'(x)$.

Similarly the
{\it left logarithmic derivative}%
\index{left logarithmic derivative} 
$\de^lf\in\Om^1(M,\mathfrak g)$ 
of a smooth mapping $f:M\to G$ is given by 
$$\de^lf.\xi_x= T_{f(x)}(\mu_{f(x)\i}).T_xf.\xi_x= (f^*\ka^l)(\xi_x)$$

\begin{theorem}{\label{nmb:2.3}}
Let $f,g:M\to G$ be smooth. Then the Leibniz rule 
holds:
$$\de^r(f.g)(x) 
     = \de^r f(x) + \Ad(f(x)).\de^r g(x).$$
Moreover, the differential form $\de^rf\in\Om^1(M;\g)$ 
satifies the `left Maurer-Cartan equation' (left because it stems 
from the left action of $G$ on itself) 
\begin{gather*}
d\de^rf(\xi,\et) - 
[\de^rf(\xi),\de^rf(\et)]^\g=0,\\
\text{ or }\quad d\de^rf - \frac12 [\de^rf,\de^rf]^\g_\wedge=0,
\end{gather*}
where $\xi,\et\in T_x M$, and 
where for $\ph\in\Om^p(M;\g),\ps\in\Om^q(M;\g)$ one puts
$$[\ph,\ps]^\g_\wedge(\xi_1,\dots,\xi_{p+q}) 
     := \frac1{p!q!}\sum_\si\on{sign}(\si)
     [\ph(\xi_{\si1},\dots),\ps(\xi_{\si(p+1)},\dots)]^\g.$$
If $h:N\to M$ is a smooth mapping, then $\de^r(f\o h)= h^*(\de^r(f))$.
If $\ph:G\to H$ is a smooth homomorphism of groups, then 
$\de^r_H(\ph\o f) = \ph'\o \de^r_G(f)$ where $\ph'=T_e\ph:\g \to \mathfrak h$.

The left logarithmic derivative 
the corresponding Leibniz rule is given by
\begin{equation*}
\de^l(fg)(x) = \de^lg(x) + 
Ad(g(x)\i)\de^lf(x)\,.
\end{equation*}
It satisfies the `right Maurer Cartan equation':
\begin{equation*}
d\de^lf + \frac12 
[\de^lf,\de^lf]^\g_\wedge=0.
\end{equation*}

\end{theorem}

For `regular Lie groups' a converse to this statement holds; see
\cite{KM97},~40.2. We shall review this in \ref{nmb:2.3} and \ref{nmb:2.4} below.
This result has a geometric interpretation in 
principal bundle geometry for the trivial principal bundle 
$\operatorname{pr_1}:M\x G\to M$ with right principal action. Then 
the submanifolds $\{(x,f(x).g):x\in M\}$ for $g\in G$ form a 
foliation of $M\x G$ whose tangent distribution is complementary to the 
vertical bundle $M\x TG \subseteq T(M\x G)$ and is invariant under the 
principal right $G$-action. So it is the horizontal distribution of the
principal connection $\om^l\in\Om^1(M\x G,\g)$ which is given by 
$\om^l(\xi_x,\et_a) := -\de^l(f.a)(\xi_x) + \ka^l(\et_a)$. 
Thus this principal connection has 
vanishing curvature which translates into the result for the right
logarithmic derivative. 

\begin{demo}{Proof}
For the Leibniz rule we compute for $\xi_x\in T_xM$, using \ref{nmb:2.1},
\begin{align*}
\de^r(f.g)(\xi_x) &= T_{f(x).g(x)}(\mu^{(f(x).g(x))\i}).T(\mu\o(f,g)).\xi_x
\\&
= T_{f(x).g(x)}(\mu^{(f(x).g(x))\i}).T\mu.(T_xf.\xi_x,T_xg.\xi_x)
\\&
= T(\mu^{f(x)\i}).T(\mu^{g(x)\i}).
\big(T(\mu_{f(x)}).T_xg.\xi_x) + T(\mu^{g(x)}).T_xf.\xi_x\big)
\\&
= T(\mu^{f(x)\i}).T(\mu_{f(x)}).T(\mu^{g(x)\i}).T_xg.\xi_x) + T(\mu^{f(x)\i}).T_xf.\xi_x
\\&
= \Ad(f(x)).\de^r g(\xi_x) + \de^rf(\xi_x).
\end{align*}
The Maurer-Cartan equation for $\de^rf $ is the pullback of the one for  $\ka^r$:
\begin{align*}
d(\de^r f) &= d(f^*\ka^r) = f^*(d\ka^r) 
= f^*(\tfrac12[\ka^r,\ka^r]_\g^\wedge)
= \tfrac12[f^*\ka^r,f^*\ka^r]_\g^\wedge
\\&
= \tfrac12[\de^rf,\de^rf]_\g^\wedge
\end{align*}
Finally, 
$\de^r(f\o \ph) = d (f\o \ph)^* \ka^r = d \ph^* f^* \ka^r = \ph^* \de^r(f)$.
For the left logarithmic derivative the proof is analogous. 
\qed\end{demo}

\subsection{\label{nmb:2.4}The 1-dimensional evolution operator}
Let $G$ be a possibly infinite dimensional Lie group with Lie algebra $\g$. 
For a closed interval $I\subset \mathbb R$  and for 
$X\in C^\infty(I,\g)$ we consider the ordinary differential equation
\begin{equation*}
\begin{cases} g(t_0)=e \\
       \p_tg(t)=T_e(\mu^{g(t)})X(t) = R_{X(t)}(g(t))
          \quad\text{ or }\ka^r(\p_tg(t)) = X(t),
\end{cases}\tag{1}
\end{equation*}
for local smooth curves $g$ in $G$, where $t_0\in I$. Then the following results hold; see \cite{KM97},~40.2:

{\it 
\begin{enumerate}
\item[(2)] Local solution curves $g$ of the differential equation 
       \thetag1 are unique.
\item[(3)]  If for fixed $X$ the differential equation \thetag1 has a 
       local solution near each $t_0\in I$, then it has also a global 
       solution $g\in C^\infty(I,G)$.  
\item[(4)]  If for all $X\in C^\infty(I,\g)$ the differential equation 
       \thetag1 has a local solution near one fixed $t_0\in I$, then 
       it has also a global solution $g\in C^\infty(I,G)$ for each 
       $X$. Moreover, if the local solutions near $t_0$ depend 
       smoothly on the vector fields $X$ then so does the global solution. 
\item[(5)]  If \thetag{4} holds, then the curve $t\mapsto g(t)\i$ is the unique local smooth curve $h$ in 
       $G$ which satifies  
$$
\begin{cases} h(t_0)=e \\
\p_th(t) = T_e(\mu_{h(t)})(-X(t)) = L_{-X(t)}(h(t)) \\
          \quad\text{ or }\ka^l(\p_th(t)) = -X(t).
\end{cases}
$$
\end{enumerate}
}

\subsection{ \label{nmb:2.5}Regular Lie groups}
If for each $X\in C^\infty(\mathbb R,\g)$ there exists 
$g\in C^\infty(\mathbb R,G)$ satisfying 
\begin{equation*}
\begin{cases} g(0)=e, \\
       \p_tg(t)=T_e(\mu^{g(t)})X(t) = R_{X(t)}(g(t))\\
          \quad\text{ or }\ka^r(\p_tg(t))= \de^rg(\partial_t) = X(t),
\end{cases}\tag{1}
\end{equation*}
then we write 
\begin{gather*}
\on{evol}^r_G(X) = \on{evol}_G(X):=g(1),\\
\on{Evol}^r_G(X)(t) := \on{evol}_G(s\mapsto tX(ts)) = g(t),
\end{gather*} 
and call it the {\it right evolution}%
\index{right evolution} 
of the curve $X$ in $G$. By \ref{nmb:2.4} the 
solution of the differential equation \thetag1 is unique, 
and for global existence it is sufficient that it has a local solution somewhere.
Then 
$$
\on{Evol}^r_G: C^\infty(\mathbb R,\g) \to \{g\in C^\infty(\mathbb R,G):g(0)=e\}
$$
is bijective with inverse the right logarithmic derivative $\de^r$.

The Lie group $G$ is called a
{\it regular Lie group}%
\index{regular Lie group} if 
$\on{evol}^r:C^\infty(\mathbb R,\g)\to G$ exists and is smooth.

We also write 
\begin{gather*}
\on{evol}^l_G(X) = \on{evol}_G(X):=h(1),\\
\on{Evol}^l_G(X)(t) := \on{evol}^l_G(s\mapsto tX(ts)) = h(t),
\end{gather*} 
if $h$ is the (unique) solution of 
\begin{equation*}
\begin{cases} h(0)=e \\
     \p_th(t) = T_e(\mu_{h(t)})(X(t)) = L_{X(t)}(h(t)),\\
          \quad\text{ or }\ka^l(\p_th(t))=\de^lh(\partial_t) = X(t).
\end{cases}\tag{2}
\end{equation*}
Clearly $\on{evol}^l:C^\infty(\mathbb R,\g)\to G$ exists and is 
also smooth if $\on{evol}^r$ does, 
since we have $\on{evol}^l(X)=\on{evol}^r(-X)\i$ by \ref{nmb:2.4}. 

Let us collect some easily seen properties of the evolution mappings.
If $f\in C^\infty(\mathbb R,\mathbb R)$, then we have 
\begin{align*}
\on{Evol}^r(X)(f(t)) &= \on{Evol}^r(f'.(X\o f))(t).\on{Evol}^r(X)(f(0)),\\
\on{Evol}^l(X)(f(t)) &= \on{Evol}^l(X)(f(0)).\on{Evol}^l(f'.(X\o f))(t).
\end{align*}
If $\ph:G\to H$ is a smooth homomorphism between regular Lie groups 
then the diagram
$$\xymatrix{ 
C^\infty(\mathbb R,\mathfrak g)   \ar[r]^{\ph'_*} \ar[d]_{\on{evol}_G}  &  
     C^\infty(\mathbb R,\mathfrak h) \ar[d]^{\on{evol}_H}\\ 
G \ar[r]^{\ph} &   H}  
$$
commutes, since 
$\p_t\ph(g(t))=T\ph.T(\mu^{g(t)}).X(t)=T(\mu^{\ph(g(t))}).\ph'.X(t)$.

Note that each regular Lie group admits an exponential mapping, 
namely the restriction of $\on{evol}^r$ to the 
constant curves $\mathbb R\to \g$. A Lie group is regular if 
and only if its universal covering group is regular.

No counter-example to the following statement is known:
\emph {All known Lie groups modelled on convenient vector spaces are regular.}
Any Banach Lie group is regular since we may consider the time 
dependent right invariant vector field $R_{X(t)}$ on $G$ and its 
integral curve $g(t)$ starting at $e$, which exists and depends 
smoothly on (a further parameter in) $X$. 
For diffeomorphism groups the evolution operator 
is just integration of time dependent 
vector fields with compact support. 


\section{The general evolution operator alias the Cartan development}\label{nmb:3}

\subsection{The space of flat differential forms}\label{nmb:3.1}
Let $G$ be a regular Lie group with Lie algebra $\g$. Let $(M,x_0)$ be a simply connected pointed possibly infinite dimensional manifold. We consider the space 
\begin{align*}
\Om^1_{\text{flat}}(M,\g) & = \{\xi \in \Om^1(M,\g): d\xi - \tfrac12 [\xi,\xi]^\g_\wedge = 0
\}
\end{align*}
of $\g$-valued 1-forms on $M$ which are \emph{flat} in the sense that they obey the left Maurer-Cartan equation.

{\bf The smooth structure on the space of flat differential forms.}
 $\Om^1_{\text{flat}}(M,\g)$ inherits the smooth structure of a Fr\"olicher space (see \cite[Section 23]{KM97}) from the convenient vector space 
$\Om^1(M,\g)$, generated by all curves $c:\mathbb R\to \Om^1_{\text{flat}}(M,\g)$ which are smooth into 
$\Om^1(M,\g)$. 
This allows us to talk about smooth mappings running through it. 
\\
Moreover, it is a Fr\"olicher space with further structure: its kinematic tangent spaces are given by 
$T_\xi \Om^1(M,\g) = \{\et\in \Om^1(M,\g): d\et - [\xi,\et]^\g_\wedge =0\}$, and the tangent bundle is again the Fr\"olicher space
$$
T\Om^1(M,\g) = \{(\xi,\et)\in \Om^1(M,\g)^2: d\xi - \tfrac12 [\xi,\xi]^\g_\wedge = 0 \text{ and  }
d\et - [\xi,\et]^\g_\wedge =0\}\,.
$$
The space $\Om^1_{\text{flat}}(M,\g)$ is a manifold in the sense of \cite{Michor84a,Michor84b} where a cartesian closed category of manifolds based on smooth curves  instead of charts as developed which up to those whose tangent spaces (proved only for finite dimensional spaces there - for Banach spaces  the same proof applies) consists exactly of the usual ones.   

{\bf Question:} If $M$ is finite dimensional, is $\Om^1_{\text{flat}}(M,\g)$ a (split) submanifold of $\Om^1(M,\g)$?
Can this be shown by a quasilinear version of Hodge theory? 


But we shall see below that $\Om^1_{\text{flat}}(M,\g)$ is diffeomorphic to $C^{\infty}((M,x_0),(G,e))$ and that it gets a Lie group structure via this map, at least in the case when $M$ is compact. 

\subsection{The space $C^\infty(M,G)$}\label{nmb:3.2}
If $M$ is a compact manifold, then $C^\infty(M,G)$ is a smooth manifold; see \cite[Section 42]{KM97}. The same is true if $M$ is locally compact for the space $C^\infty_c(M,G)$ of smooth mappings which equal the constant $e$ off some compact subset of $M$; let us call these smooth mappings with compact support. But in general, there is no smooth structure admitting an atlas on $C^\infty(M,G)$, since the space $C^\infty(M,N)$ is not locally contractible in its natural topology.
Thus we consider $C^\infty(M,G)$ as a Fr\"olicher Lie group with pointwise group operations in the general situation; see \cite[Section 23]{KM97}: 
The    canonical smooth structure is described by 
    \[
      C^\infty(M,G)\xrightarrow{C^\infty(c,f)} C^\infty(\mathbb R,\mathbb R) 
      \xrightarrow{\la} \mathbb R
    \]
    where $c\in C^\infty(\mathbb R,M)$, where $f$ is in 
    $C^\infty(G,\mathbb R)$ or in a generating set of functions, and 
    where $\la\in C^\infty(\mathbb R,\mathbb R)'$.
It is a Fr\"olicher space with a natural tangent bundle 
$$
TC^\infty(M,G)= C^\infty(M,TG)\xrightarrow{(\pi_G,\ka^r)_*} C^\infty(M,G)\x C^\infty(M,\g)\,.
$$
The space $C^{\infty}(M,G)$ is a manifold based on smooth curves instead of charts as developed in \cite{Michor84a,Michor84b}.  

\begin{theorem}\label{nmb:3.3}
Let $G$ be a regular Lie group with Lie algebra $\g$. Let $(M,x_0)$ be a simply connected pointed possibly infinite dimensional  manifold.
In general, there exists a unique smooth mapping, called \emph{evolution operator}, 
$$
\on{Evol}=\on{Evol}^M_G: \Om^1_{\text{flat}}(M,\g) \to C^{\infty}(M, G)
$$
which satisfies 
$$
\de^r \o \on{Evol} = \on{Id} \text{ and  } \on{Evol}(x_0) = e.
$$
$\Evol$ is a natural transformation $\Om^1_{\text{flat}}(\;, \on{Lie}(\;))\to C^{\infty}(\;,\;)$ between the two contra-covariant functors 
from the category of simply connected smooth manifolds times the category of Lie groups into the category of Fr\"olicher spaces:
For smooth $h:(N,y_0)\to (M,x_0)$ and $\ph:G\to H$ a smooth homomorphism between regular Lie groups with 
$T_e\ph = \ph':\g\to \mathfrak h$ we have 
$$
C^{\infty}(h,\ph)\o \on{Evol}_G^{(M,x_0)} = \on{Evol}^{(N,y_0)}_H \o \Om^1_{\text{flat}}(h,\ph')\,;
$$
in detail $\on{Evol}(h^*\xi) = \on{Evol}(\xi)\o h$  and $\ph\o \on{Evol}_G(\xi) = \on{Evol}_H(\ph'\o \xi)$.
\end{theorem}

If one wants to avoid choosing a point $x_0$ in $M$, then $\on{Evol}$ exists and is unique up to right translations in $G$.
Note that a Lie group $G$ is regular if the theorem holds for $(M,x_0)= (\mathbb R,0)$. 
See  \cite[40.2]{KM97} for the main part of the proof. For finite dimensional $M$ and Lie groups a proof can be found in 
\cite{Onishchik61,Onishchik64,Onishchik67}, 
or in \cite{Cartan35,Cartan37} or  \cite{Griffiths74} (proved with moving 
frames); see also \cite[5.2]{AlekseevskyMichor95b}.

\begin{proof} See \cite[40.2 using 39.1 and 39.2]{KM97} with a small gap for the first part. 
For completeness' sake we repeat the proof in the more simple situation here.

If we are given a 1-form $\xi\in\Om^1(M,\g)$ with 
$d\xi - \tfrac12[\xi,\xi]^\g_\wedge=0$ then we consider the 
1-form $\om^\xi\in\Om^1(M\x G,\g)$ given by
\[
\om^\xi = \ka^l - (\Ad\o\nu).\xi\,,\quad
\om^\xi_{(x,g)}(Y_x, T\mu_g. X) = X - \Ad(g\i).\xi_x(Y_x)
\]
for $Y_x\in T_xM$ and $X\in \g$. Then $\om^\xi$ is a principal connection form on $M\x G$, since it 
reproduces the generators in $\g$ of the fundamental vector fields 
for the principal right action, i.e., the left invariant vector 
fields $0\x L_X$, and $\om^\xi$ is $G$-equivariant:
\begin{align*}
((\mu^g)^*\om^\xi)_{(x,h)} &= \om^\xi_{(x,hg)}\o (\Id_{TM}\x T(\mu^g)) = 
     T(\mu_{g\i h\i}).T(\mu^g) - \Ad(g\i h\i).\xi_x\\
&= \Ad(g\i).\big(\ka^l_h - \Ad(h\i).\xi_x\big) = \Ad(g\i).\om^\xi_{(x,h)}.
\end{align*}
Since the structure group $G$ is regular, for each smooth curve $c:\mathbb R\to M$, the smooth lift of $c$ given by 
$$
t\mapsto \big(c(t), \on{Evol}_G^r(-\om^\xi(c',0))(t).u\big) = (c(t),g(t)) =: \on{Pt}^\xi(c,t,u)
$$
defines a mapping $\on{Pt}^\xi(c,t,\;): \{c(0)\}\x G \to \{c(t)\}\x G$ which has all the usual properties of parallel transport, namely:
(1) It is horizontal since 
\begin{align*}
g'(t) &= \p_t  \on{Evol}_G^r(-\om^\xi(c',0))(t).u = -T\mu^{g(t)}.\om^\xi(c',0)\quad\text{ by \ref{nmb:2.5} (1),}
\\&
= 0 + T\mu_{g(t)} \Ad(g(t)\i).\xi(c'(t))
\\
\om^\xi(c'(t),g'(t)) &=  \Ad(g(t)\i).\xi(c'(t)) - \Ad(g(t)\i).\xi(c'(t)) = 0\,.
\end{align*}
(2) It is $G$-equivariant for the principal right action. (3) For smooth $f:\mathbb R\to\mathbb R$ we have $\on{Pt}^\xi(c, f(t),u) = \on{Pt}^\xi(c\o f,t, \on{Pt}^\xi(c,f(0),u))$. (4) It is smooth as a mapping $\on{Pt}^\xi:C^{\infty}(\mathbb R,M)\x_{\on{ev}_0,M,\on{pr}_3}(\mathbb R\x M\x G) \to M\x G$ where $C^{\infty}(\mathbb R, M)$ carries the Fr\"olicher structure, corresponding to the compact $C^{\infty}$-topology.
(5) $\on{Pt}^\xi$ depends also smoothly on the choice of $\xi\in \Om^1_{\text{flat}}(M,\g)$; this was not checked in  \cite[39.1]{KM97}. A  smooth curve $s\mapsto \xi(s)\in \Om^1_{\text{flat}}(M,\g)$ leads to a smooth curve $s\mapsto\on{Pt}^{\xi(s)}(c,t,u)\in M\x G$ by looking at the explicit formulas above.

The connection $\om^\xi$ is flat since
\begin{align*}
d\om^\xi+\tfrac12[\om^\xi,\om^\xi]_\wedge 
&= d\ka^l + \tfrac12[\ka^l,\ka^l]_\wedge 
     - d(\Ad\o\nu) \wedge \xi - 
     (\Ad\o\nu). d\xi\\
&\quad - [\ka^l,(\Ad\o\nu). \xi]_\wedge + 
     \tfrac12[(\Ad\o\nu).\xi,
     (\Ad\o\nu). \xi]_\wedge\\
&= - (\Ad\o\nu). (d\xi- 
     \tfrac12[\xi,\xi]_\wedge) = 0\,.
\end{align*}
Since the structure group $G$ is regular, by theorem \cite[39.2]{KM97} 
the horizontal bundle $\ker(\om^\xi)\subset T(M\x G)$ is integrable: For each $x\in M$ let $u:U\to E$ be a chart with $x\in U$ and $u(x)=0$ such that $u(U)\subset E$ is a star-shaped $c^\infty$-open set in $E$. For $y\in U$ let $c_y:[01]\to U$ be given by $c_y(t)= u\i (t,u(y))$. Then $\ps:U\to M\x G$ given by $\ps(y)= \on{Pt}^\xi(c_y,1,g)$ is a smooth section $U\to M\x G$ through $(x,g)$ which is horizontal thus an integral submanifold of $\ker \om^\xi$; for the quite lengthy proof of horizontality see the proof of theorem \cite[39.2]{KM97}.

The projection $\pr_1 :M\x G\to M$, restricted to each horizontal 
leaf, is a covering. Thus, it may be inverted over the simply 
connected manifold $M$, and the inverse $(\Id_M,f):M\to M\x G$ is a horizontal section, i.e., $T(\Id_M,f):TM \to \ker(\om^\xi)$ is an isomorphism. 
Therefore 
\begin{gather*}
0=((\Id,f)^*\om^\xi)_x =  (f^*\ka^l - f^*(\Ad \o\nu).\xi)_x =( \de^lf)_{x} - \Ad(f(x)\i).\xi_x
\\
(\de^r f)_x = \Ad(f(x))(\de^l f)_x = \xi_x
\end{gather*}
 for $x\in M$, as required.
Moreover, $(\Id_M\x f)$ is
unique up to the choice of the branch of the covering and the choice 
of the leaf, i.e., $f$ is unique up to a right translation by an 
element of $G$. We may fix $f =: \Evol^M_G(\xi)$ by stipulating $f(x_0)= e$. Moreover, $\Evol^M_G(\xi)$ depends smoothly on $\xi$ since the parallel transport $\on{Pt}^\xi$ depends smoothly on $\xi$.

It remains to check that $((M,x_0),G)\mapsto \Evol^{(M,x_0)}_G$ is a natural transformation. For $h:(N,y_0)\to (M,x_0)$ we have, using \ref{nmb:2.3}
\begin{align*}
\de^r \Evol^{(N,y_0)}_G(h^*\xi) &= h^*\xi = h^*\de^r \Evol^{(M,x_0)}(\xi) = \de^r(\Evol^{(M,x_0)}(\xi)\o f)\,,
\end{align*}
thus $\Evol^{(N,y_0)}_G(h^*\xi) =\Evol^{(M,x_0)}(\xi)\o f$ since both sides agree on $y_0$.
If $\ph:G\to H$ is a smooth homomorphism, then $\ph^* \ka^r_H = \ph' . \ka^r_G$ 
since 
\begin{align*}
(\ph^*\ka^r_H)_g(T_e\mu^g.X) &= (\ka^r_H)_{\ph(g)}(T_g\ph.T_e\mu^g.X) = T\mu^{\ph(g)\i}.T_g\ph.T_e\mu^g.X
\\&
=T_e\ph.X = \ph'.T_g\mu^{g\i}.T_e\mu^g.X =\ph'.(\ka^r_G)_g(T_e\mu^g.X)\,,
\end{align*}
and obviously
$\ph'_*: \Om^1_{\text{flat}}(M,\g) \to \Om^1_{\text{flat}}(M,\mathfrak h) $, thus by \ref{nmb:2.3} again we get
\begin{align*}
\de^r (\ph\o \Evol^M_G(\xi)) &= \ph'.\de^r \Evol^M_G(\xi) = \ph'.\xi = \de^r \Evol^M_H(\ph'.\xi)
\end{align*}
thus $\ph\o \Evol^M_G(\xi)=\Evol^M_H(\ph'.\xi)$ since both sides map $x_0\in M$ to $e\in H$.
\end{proof}

\begin{theorem}\label{nmb:3.4} For a regular Lie group $G$ and simply connected $(M,x_0)$ we have for 
$\xi,\et\in \Om^1_{\text{flat}}(M,\g)$:
\begin{gather*}
\Evol(\xi).\Evol(\et) = 
     \Evol\Bigl(\xi+(\Ad_G\o \Evol(\xi)).\et\Bigr),\\
\Evol(\xi)\i = \Evol\Bigl((-\Ad_G\o \nu\o\Evol(\xi)).\xi\Bigr),
\end{gather*}
so that $\Evol:\Om^1_{\text{flat}}(M,\g) \to C^\infty((M,x_0),(G,e))$ is a bijective smooth 
homomorphism of Fr\"olicher Lie groups, where  $C^\infty((M,x_0),(G,e))$ carries the pointwise group operations and where
on $\Om^1_{\text{flat}}(M,\g)$  
the operations are given by 
\begin{align*}
(\xi*\et)(x) &= \xi(x)+\Ad_G(\Evol(\xi)(x)).\et(x),\\
\xi\i(x) &= -\Ad_G(\Evol(\xi)(x)\i).\xi(x).
\end{align*}
With these operations and with 0 as unit element 
$(\Om^1_{\text{flat}}(M,\g),*)$ becomes a regular Fr\"olicher Lie group isomorphic to $C^\infty((M,x_0),(G,e))$.
Its Lie algebra is $T_0\Om^1_{\text{flat}}(M,\g) = \{\et\in\Om^1(M,\g)): d\et =0\}=: Z(M,\g)$ with bracket  
$$
[\xi_1,\xi_2]^{Z(M,\g)} = [\xi_1,d\i\xi_2]^\g + [d\i\xi_1,\xi_2]^\g, \quad \xi_i\in Z(M,\g)\,,
$$
where 
$$d\i : \{\xi\in\Om^1(M,\g): d\xi = 0\} =: Z(M,\g) \to C^{\infty}((M,x_0),(\g,0))$$ 
is the bounded operator of the Poincar\'e lemma, the inverse of the exterior derivative.
\end{theorem}

This formula for the Lie bracket fits nicely to the 1-dimensional version derived in \cite[38.12]{KM97}.

\begin{proof} We have $\mu^*\ka^r = \pr_1^*\ka^r + (\Ad\o\pr_1)\pr_2^*\ka^r$ because
\begin{align*}
(\mu^*\ka^r)_{(a,b)}(\xi_a,\et_b) &= \ka^r_{\mu(ab)}\big(T\mu.(\xi_a,\et_b)\big)
= \ka^r_{ab}\big(T_a\mu^b.\xi_a + T_b\mu_a.\et_b\big)
\\&
= T_{ab} \mu^{(ab)\i}\big(T_a\mu^b.\xi_a + T_b\mu_a.\et_b\big)
\\&
= T \mu^{a\i} .T\mu^{b\i}.T\mu^b.\xi_a + T \mu^{a\i} .T\mu^{b\i}.T\mu_a.\et_b
\\&
= \ka^r_a(\xi_a) + \Ad(a).\ka^r_b(\et_b).
\end{align*}
For $\xi,\et\in \Om^1_{\text{flat}}(M,\g)$ we have therefore
\begin{align*}
\de^r&\big(\Evol(\xi).\Evol(\et)\big) = \big(\mu\o(\Evol(\xi),\Evol(\et))\big)^*\ka^r
\\&
= (\Evol(\xi),\Evol(\et))^*\mu^*\ka^r
= (\Evol(\xi),\Evol(\et))^*\big( \pr_1^*\ka^r + (\Ad\o\pr_1)\pr_2^*\ka^r\big)
\\&
= \Evol(\xi)^*\ka^r + (\Ad\o \Evol(\xi)).\Evol(\et)^*\ka^r = \xi + \Ad(\Evol(\xi))\et.
\end{align*}
which implies 
\[
\Evol(\xi).\Evol(\et) = \Evol(\xi*\et),\quad \Evol(\xi)\i=\Evol(\xi\i).
\]
Thus, $\Evol: \Om^1_{\text{flat}}(M,\g)\to C^\infty(M,G)$ is a group 
isomorphism onto the subgroup $ C^\infty((M,x_0),(G,e)):=\{f\in C^\infty(M,G):f(x_0)=e\}$
with the pointwise product, 
which, however, is only a Fr\"olicher space in general, see \cite[23.1]{KM97}. If $M$ is compact then $C^\infty((M,x_0),(G,e))$ is smooth regular Lie group. If $M$ is finite dimensional, then one has to refine the topology (control near infinity) to make it into a disjoint union
of regular Lie groups, and then one has to mimick this procedure also on $\Om^1_{\text{flat}}(M,\g)$. 

In general, we will just take both groups as Fr\"olicher spaces with special properties (having a tangent bundle, e.g.).

It follows that the product on $\Om^1_{\text{flat}}(M,\g)$ has the  properties of a group structure. 

A direct proof is fun and needs naturality of $\Evol$ in an essential way. As entertainment we compute the following:
\begin{align*}
\xi\i * \xi &= - \Ad_G(\Evol(\xi)\i).\xi + \Ad_G\big(\Evol(-\Ad_G(\Evol(\xi))\i).\xi\big)\xi = 0 \text{ since }
\\&
(\Evol(\xi)\i)^*\ka^r = (\nu\o\Evol(\xi))^*\ka^r = \Evol(\xi)^*\nu^*\ka^r 
\\&\qquad= \Evol(\xi)^*(-(\Ad\o\nu)\ka^r) = -\Ad(\Evol(\xi)\i)\Evol(\xi)^*\ka^r 
\\&\qquad =  -\Ad(\Evol(\xi)\i)\xi \text{ implies again the expression for the inverse}
\\&
\Evol(\xi)\i = \Evol(-\Ad(\Evol(\xi)\i)\xi)\,.
\end{align*} 

Now we aim for the Lie bracket. Since $\de^r: C^\infty((M,x_0),(G,e)) \to \Om^1_{\text{flat}}(M,\g) \subset \Om^1(M,\g)$ is is a smooth group isomorphism of Fr\"olicher Lie groups, we can just take its tangent mapping at the constant $e$ which will become a homomorphism of Lie algebras.
As an aside note that $\de^r$ is a Lie derivative analogous to the finite dimensional version in \cite[Chapter IX]{KMS93}.  The infinite dimensional version has been worked out in detail in \cite[12.2--12.5]{Michor80} for mappings $f:M\to N$; here the only difference is that the forms are $\g$-valued and that we make use the right trivilization of $TG$.
So we choose a smooth curve $t\mapsto f(t)\in C^\infty((M,x_0),(G,e))$ with $f(0)=e$ and $\p_t|_0 f(t) =X\in C^{\infty}(M,\g)$ with $X(x_0)=0$; then we consider the smooth mapping $\hat f:\mathbb R\x M \to G$. By the Maurer-Cartan equation \ref{nmb:2.3}, for a vector field $Y\in \X(M)$ we have 
\begin{align*}
0&=d(\de^r\hat f)((\p_t,0_M),(0_{\mathbb R},Y)) - \big[(\de^r\hat f)(\p_t,0_M), (\de^r\hat f)(0_{\mathbb R},Y)\big]^\g
\\&
= \p_t\big((\de^r\hat f)(0_{\mathbb R},Y)\big) - \L_Y\big(T\mu^{f\i}\p_t f\big) - (\de^r\hat f)([(\p_t,0_M),(0_{\mathbb R},Y)]^{\X(\mathbb R\x M)})
\\&\qquad
- \big[(\de^r\hat f)(\p_t,0_M), (\de^r\hat f)(0_{\mathbb R},Y)\big]\,.
\end{align*}
Choosing $t=0$ and using $f(0,x)=e$ this becomes
$$
(T_e\de^r.X)(Y) = \L_YX = dX(Y). 
$$
Now we can write down the Lie bracket. Since $M$ is simply connected, its de~Rham cohomology $H^1(M)=0$, which also holds for 
$\g$-valued 1-forms and for infinite dimensional $M$; see \cite[Section 34]{KM97}. Let 
$$d\i : \{\xi\in\Om^1(M,\g): d\xi = 0\} =: Z(M,\g) \to C^{\infty}((M,x_0),(\g,0))$$ 
be the bounded (by lemma \ref{nmb:3.5} below) operator of the Poincar\'e lemma; i.e., the inverse of exterior derivative. Then for $X_i\in C^{\infty}((M,x_0),(\g,0))$ we have
\begin{align*}
T_e\de^r.[X_1,X_2]^\g &= d [X_1,X_2]^\g = [dX_1,X_2]^\g + [X_1,dX_2]^\g,\quad \text{ thus }
\\
[\xi_1,\xi_2]^{Z(M,\g)} &= [\xi_1,d\i\xi_2]^\g + [d\i\xi_1,\xi_2]^\g, \text{ for } \xi_i\in Z(M,\g). \qedhere
\end{align*}
\end{proof}

\begin{lemma}\label{nmb:3.5} 
For a simply connected  manifold $M$ the inverse of the exterior derivative $d\i:Z(M,\g) \to C^{\infty}((M,x_0),(\g,0))$ is a bounded linear operator.
\end{lemma}

\begin{proof}
For a star-shaped $c^{\infty}$-open subset $M$ in a convenient vector space $E$ this follows from the explicit formula for the Poincar\'e operator $d\i \xi(x) = \int_0^1t\om(tx)(x)dt$.  In general, we cover the simply connected manifold $M$ by open charts
$u:U\to u(U)\subset E$ with $x_0\in U$, $u(x_0)=0$, and  $u(U)$ star-shaped in $E$.  Via the diffeomorphisms $u^*$ the operators 
$d\i: Z(U,\g) \to C^{\infty}((U,x_0),(\g,0))$ are all bounded. Since $C^{\infty}((M,x_0),(\g,0))$ and  $Z(M,\g)$ both carry the initial structure with respect to the restriction operators to $U$, the result follows. 
\end{proof}

\begin{theorem}\label{nmb:3.6}
For a regular Lie group and $\xi\in \Om^1_{\text{flat}}(M,\g)$ we consider $f= \Evol(\xi):(M,x_0)\to (G,e)$.
For a mapping $h:M\to \g$ we consider
$\et:= \Ad(f)d h\in\Om^1(M,\g)$. Then we have $d\et - \frac12[\xi,\et]^\g_\wedge =0$.

If conversely $\et\in\Om^1(M,\g)$ satisfies $d\et - \frac12[\xi,\et]^\g_\wedge =0$ then there exists  a unique smooth $h:(M,x_0)\to (\g,0)$ such that $dh = \Ad(f\i)\et$.
\end{theorem}

\begin{proof} If $h:M\to \g$ exists then $d^2h = 0$ so that 
\begin{align*}
d\et &= d\Ad.Tf\wedge dh + \Ad(f)d^2h = (\ad\o \ka^r.Tf)\Ad(f)dh + 0 
\\&
=  \tfrac12 [f^*\ka^r,\et]^\g_\wedge = \tfrac12 [\xi,\et]^\g_\wedge. 
\end{align*}
Conversely, by this computation  $d\et - \frac12[\xi,\et]^\g_\wedge =0$ implies that the $\g$-valued 1-form 
$\be:= \Ad(f\i)\et$ is closed. Since $M$ is simply connected, $\be = dh$ for $h\in C^{\infty}(M,\g)$ which is unique up to addition of a constant.
\end{proof}

\begin{theorem}\label{nmb:3.7} Let $G$ be a Lie group.
Then via right trivialization $(\ka^r,\pi_G):TG\to \g\x G$ 
the tangent group $TG$ is isomorphic to the semidirect product 
$\g\rtimes G$, where $G$ acts by $\Ad:G\to\on{Aut}(\g)$. So for $g,h\in G$ and $X,Y\in \g$ we have:
\begin{align*}
\mu_{\g\rtimes G}((X,g),(Y,h)) &= (X + \Ad(g)Y,gh)\,\tag{\nmb:{1}}\\
\nu_{\g\rtimes G}(X,g) &= (- \Ad(g\i)X,g\i)\,,\\
[(X_1,Y_1), (X_2,Y_2)]_{\g\rtimes\g} &= ([Y_1,X_2]-[Y_2,X_1], [Y_1,Y_2])\,,
\\
\Ad^{\g\rtimes G}_{(X,g)}(Y,Z) &= (\Ad(g)Y - [\Ad(g)Z,X],\Ad(g)Z)
\end{align*}
If $G$ is a regular Lie group, then so is $TG\cong \g\rtimes G$ 
and $T\evol^{\mathbb R}_G$ corresponds to $\evol^{\mathbb R}_{TG}$, via
\[
\xymatrix@C+0.8cm{
TC^\infty(\mathbb R,\g) \ar[d]_{T\evol^{\mathbb R}_G } \ar[r]^{\cong} & 
C^\infty(\mathbb R,\g\rtimes\g) \ar[ld]^{\evol^{\mathbb R}_{TG} } 
     \ar[d]^{\evol^{\mathbb R}_{\g\rtimes G}} \\
TG \ar[r]_{\cong} & \g\rtimes G\,.
}\tag{\nmb:{2}}
\]
In particular, for 
$(Y,X)\in C^\infty(\mathbb R,\g\x \g)=TC^\infty(\mathbb R,\g)$, where $X$ 
is the footpoint, we have 
\begin{gather*}
\evol^{\mathbb R}_{\g\rtimes G}(Y,X) = 
\Bigl(\Ad(\evol_G^{\mathbb R}(X))\int_0^1\Ad(\Evol^{\mathbb R}_G(X)(s)\i).Y(s)\,ds,\; 
     \evol^{\mathbb R}_G(X)\Bigr)\\
T_X\evol^{\mathbb R}_G.Y = 
T(\mu_{\evol^{\mathbb R}_G(X)}).\int_0^1\Ad(\Evol^{\mathbb R}_G(X)(s)\i).Y(s)\,ds,
\tag{\nmb:{3}}\\
T_X(\Evol^{\mathbb R}_G(\var)(t)).Y = T(\mu_{\Evol^{\mathbb R}_G(X)(t)}).
     \int_0^t\Ad(\Evol^{\mathbb R}_G(X)(s)\i).Y(s)\,ds.
\end{gather*}
For a pointed simply connected manifold $(M,x_0)$ we have 
\begin{align*}
T\Om^1_{\text{flat}}(M,\g) &= \{(\xi,\et)\in \Om^1(M,\g)^2: d\xi - \tfrac12 [\xi,\xi]^\g_\wedge = 0 \text{ and  }
d\et - [\xi,\et]^\g_\wedge =0\}
\\&
\cong  \{(\et,\xi)\in \Om^1(M,\g\rtimes \g): (d\et - [\xi,\et]^\g_\wedge, d\xi - \tfrac12 [\xi,\xi]^\g_\wedge) = (0,0) \}
\tag{\nmb:{4}}\\&
= \Om^1_{\text{flat}}(M,\g\rtimes\g)
\end{align*}
and the following diagram commutes:
\[
\xymatrix@C+0.8cm{
T\Om^1_{\text{flat}}(M,\g) \ar[r]_{T\Evol^{M}_G \quad} \ar[d]_{\cong} & 
TC^{\infty}((M,x_0),(G,e)) \ar[d]^{\cong} \\
\Om^1_{\text{flat}}(M,T\g) \ar[d]_{\cong} \ar[r]^{\Evol^M_{TG}\quad} & C^{\infty}((M,x_0),(TG,0_e))  \ar[d]^{\cong} \\
\Om^1_{\text{flat}}(\mathbb R,\g\rtimes\g) 
     \ar[r]^{\Evol^{M}_{\g\rtimes G} \qquad} 
 &  C^{\infty}((M,x_0),(\g\rtimes G,(0,e)))\,.
}\tag{\nmb:{5}}
\]
\end{theorem}

\begin{demo}{Proof} The first part may be found in  \cite[38.10]{KM97}. We repeat the derivation of the formulas. By \ref{nmb:2.1},  for $g,h\in G$ and $X,Y\in \g$, we have
\begin{align*}
T_{(g,h)}\mu.(R_X(g)&,R_Y(h)) = T(\mu^h).R_X(g) + T(\mu_g).R_Y(h) \\
&= T(\mu^h).T(\mu^g).X + T(\mu_g).T(\mu^h).Y 
= R_X(gh) + R_{\Ad(g)Y}(gh), \\
T_{g}\nu.R_X(g) &= - T(\mu^{g\i}).T(\mu_{g\i}).T(\mu^g).X 
= - R_{\Ad(g\i)X}(g\i),
\end{align*}
so that $\mu_{TG}$ and $\nu_{TG}$, and the Lie bracket, after right trivialization,  are given by
\begin{align*}
\mu^{\g\rtimes G}((X,g),(Y,h)) &= (X + \Ad(g)Y,gh)\tag{\nmb:{6}}\\
T\mu^{\g\rtimes G}_{(X,g)}(Y',h') &= (\Ad(g)Y',T\mu_gh'), \quad h'\in T_hG\\
\nu^{\g\rtimes G}(X,g) &= (- \Ad(g\i)X,g\i),\\
(X,g).(Y,h).(X,g)\i &= (X + \Ad(g)Y,gh).(-\Ad(g\i)X,g\i)
\\&
= (X + \Ad(g)Y - \Ad(ghg\i)X, ghg\i)
\\
\Ad^{\g\rtimes G}(X,g)(Y',h') &= (\Ad(g)Y' - \Ad(e)[\Ad(g)h',X],\Ad(g)h'])
\\
\ad^{\g\rtimes G}(X',g')(Y',h') &= ([g',Y'] - [[g',h'],0]) - [h',X'],[g',h']
\\
[(X_1,Y_1), (X_2,Y_2)]_{\g\rtimes\g} &= ([Y_1,X_2]-[Y_2,X_1], [Y_1,Y_2])\,.
\end{align*}
That diagram \thetag{\nmb|{2}} commutes and equations \thetag{\nmb|{3}} hold, has been proven in \cite[38.10]{KM97}.
Now we prove that diagram \thetag{\nmb|{5}} commutes. For the bottom square this follows from theorem \ref{nmb:3.3}.
We consider a curve $t\mapsto \xi(t)\in \Om^1_{\text{flat}}(M,\g)$ which is smooth into $\Om^1(M,\g)$. Then 
$\p_t|_0 \xi(t) =: \et\in T_{\xi(0)}\Om^1_{\text{flat}}(M,\g)$ so that $d\et - [\xi(0),\et]^{\g}_\wedge = 0$. 
As in the beginning of the proof of theorem \ref{nmb:3.3} we now consider the smooth curve of flat principal connections 
\[
t\mapsto \om^{\xi(t)} = \ka^l - (\Ad\o \nu).\xi(t) 
\]
on $M\x G$, and we let $L^{\xi(t)}\subset M\x G$ be the horizontal leaf through $(x_0,e)$ for the connection 
$\om^{\xi(t)}$. Recall that 
$$(\Id_M,\Evol^M_G(\xi(t))) = (\pr_1|L^{\xi(t)})\i:M\to L^{\xi(t)}\subset M\x G$$
is the horizontal lift.  
Likewise, writing $\xi=\xi(0)$, 
$$(\Id_M,\Evol^M_{\g\rtimes G}(\et,\xi) = (\pr_1|L^{(\et,\xi)})\i:M\to L^{(\et,\xi)}\subset M\x (\g\rtimes G)$$
where $L^{(\et,\xi)}$ is the horizontal leaf through $(x_0,0,e)$ of the principal connection 
$$\om^{(\et,\xi)}
= \ka^{l,\g\rtimes G} - (\Ad^{\g\rtimes G}\o \nu^{\g\rtimes G})(\et,\xi)\,.$$

\noindent{\bf Claim.} $\p_t|_0 L^{\xi(t)}$ is the horizontal leaf $L^{(\et,\xi)}$ of the flat principal connection $\om^{(\et,\xi(0))}$ on the principal bundle $M\x (\g\rtimes G)\to M$. This is sufficient to finish the proof.

$T_{(x_0,e)}L^{\xi(t)}$ consists of all 
$(Y,X(t))\in T_{x_0}M\x \g$ such that 
$\om^{\xi(t)}_{(x_0,e)}(Y,X(t)) = X(t) - \xi(t)(Y) = 0$; since $L^{\xi(t)}$ is a horizontal leaf we may fix $Y$. Then 
$$\p_t|_0 \om^{\xi(t)}_{(x_0,e)}(Y,X(t)) = X' - \et(Y) = 0\,.$$
On the other hand 
$\om^{(\et,\xi)}_{(x_0,(0,e))}(Y,(Z,X)) = (Z-\et(Y),X-\xi(Y))$
since for $(X,g)\in \g\rtimes G$ and $(X',g')\in T_X\g\times T_gG$ we have 
\begin{align*}
&\om^{(\et,\xi)}_{(x,(X,g))}\big(Y,(X,X';g,g')\big) 
= (T\mu^{\g\rtimes G}_{(X,g)\i})(X',g') - \Ad^{\g\rtimes G}((X,g)\i)(\et(Y),\xi(Y))
\\&
= (T\mu^{ \g\rtimes G}_{(-\Ad(g\i)X,g\i)})(X',g') - \Ad^{\g\rtimes G}(-\Ad(g\i)X,g\i)(\et(Y),\xi(Y))
\\&
= (\Ad(g\i)X',T\mu_{g\i}g') 
\\&\qquad
- (\Ad(g\i)\et(Y)-[\Ad(g\i)\xi(Y),-\Ad(g\i)X],\Ad(g\i)\xi(Y))
\\&
= \big(\Ad(g\i)X' - \Ad(g\i)\et(Y)- \Ad(g\i)[\xi(Y),X],T\mu_{g\i}g'-\Ad(g\i)\xi(Y)\big). \!\qed
\end{align*}
\end{demo}

\def\cprime{$'$}

\end{document}